\documentclass[letterpaper,10pt]{amsart}
\usepackage[all]{xy}                      %
  
\CompileMatrices                            

\UseTips                                    

\input xypic
\usepackage[bookmarks=true]{hyperref}       

\usepackage{amssymb,latexsym,amsmath,amscd}
\usepackage{xspace}
\usepackage{color}
\usepackage{graphicx}
\usepackage{dsfont}

\reversemarginpar

\theoremstyle{plain}
\newtheorem{theorem}{Theorem}[section]
\newtheorem*{theorem*}{Theorem}

\theoremstyle{definition}
\newtheorem{definition}[theorem]{Definition}

\newcommand{\enm}[1]{\ensuremath{#1}}          %

\newcommand{\cal}[1]{\mathcal{#1}}

\newcommand{\PP}{\enm{\mathbb{P}}}

\newcommand{\Ii}{\enm{\cal{I}}}

\newcommand{\Oo}{\enm{\cal{O}}}

\newcommand{\Uu}{\enm{\cal{U}}}

\renewcommand{\phi}{\varphi}
\renewcommand{\theta}{\vartheta}
\renewcommand{\epsilon}{\varepsilon}

\renewcommand{\to}[1][]{\xrightarrow{\ #1\ }}

\newcommand{\old}[1]{}

\begin{document}

\title[extensions of subvarieties]
{Extending infinitely many times arithmetically Cohen-Macaulay and Gorenstein subvarieties of projective spaces}
\author{E. Ballico}
\address{Dept. of Mathematics\\
 University of Trento\\
38123 Povo (TN), Italy}
\email{edoardo.ballico@unitn.it}
\thanks{The author was partially supported by MIUR and GNSAGA of INdAM (Italy).}

\begin{abstract}
We give examples of infinitely extendable (not as cones) arithmetically Cohen-Macaulay and arithmetically Gorenstein subvarieties of projective spaces and which are not complete intersections.
The proof uses the computation of the dimension of the Hilbert scheme of codimension $2$ subschemes of projective spaces due to G. Ellingsrud and of arithmetically Gorenstein codimension $3$ subschemes due to J. O. Kleppe and R.-M. Mir\'{o}-Roig.
\end{abstract}

\maketitle

\section{Introduction}
We recall the following classical definition (classical at least since \cite{bvdv,sw,t,XXX,z1,z2}).

\begin{definition}\label{a1}
Let $X\subset \PP^n$ be an integral and non-degenerate variety. See $\PP^n$ as a hyperplane of $\PP^{n+1}$. We say that $X$ is extendable if there is an integral and non-degenerate variety $X_1\subset \PP^{n+1}$ such that $X_1\cap H =X$ (scheme-theoretic intersection) and $X_1$ is not a cone with base $X$ and vertex  $p\in \PP^{n+1}\setminus H$. In this case we say that $X_1$ is an extension or an $1$-extension of $X$. Fix an integer $r\ge 2$. We say that $X$ is $r$-extendable if there are integral and non-degenerate varieties $X_i\subset \PP^{n+i}$, $1\le i\le r$, with $X_1$ an extension of $X$ and $X_i$, $2\le i\le r$, an extension of $X_{i-1}$. We say that $X$ is infinitely extendable if it is $r$-extendable for all positive integers $r$.
If $X$ is smooth (resp. locally a complete intersection) we say that $X$ is smoothly $r$-extendable (resp. $r$-extendable as locally complete intersection) if it we may take smooth (resp. locally a complete intersection) the varieties $X_1,\dots ,X_r$ appearing in the definition of $r$-extendability.
\end{definition}

Complete intersections are infinitely extendable (just extends their minimal equations) and smooth complete intersections are infinitely smoothly extendable.

Let $r$ be any positive integer. A vector bundle $E$ on $\PP^n$ is said to be $r$-extendable if there are a degree one embedding $j: \PP^n\hookrightarrow \PP^{n+r}$ and a vector bundle $F$ on $\PP^{n+r}$ such that $E \cong j^\ast (F)$. $E$ is said to be infinitely extendable if it is $r$-extendable for all $r$. A. N. Tyurin proved that a finite rank vector bundle on $\PP^n$ is infinitely extendable if and only if it is a direct sum of line bundles (\cite{coa1,coa2,ct,s1,s2,t}) and his result was extended to other varieties instead of a projective space (\cite{pt}). Results of this type  are called Babylonian towers. Zero-loci with codimension $a$ of rank $a$ vector bundles on $\PP^n$ are locally a complete intersections. H. Flenner proved that locally complete intersections which are infinitely extendable as locally complete intersections are complete intersections (\cite[Theorem 2]{coa2}, \cite{f}). The non-existence (except complete intersections) of infinitely extendable locally complete intersection (except global complete intersection) easily implies the corresponding result for vector bundles.

\begin{theorem}\label{aa1}
Fix integers $n\ge x+2 \ge 4$. There are infinitely extendable arithmetically Cohen-Macaulay integral schemes $X\subset \PP^n$ with $X$ of codimension $x$ and not a complete intersection. For $x\ge 3$ there is $X$ which is arithmetically Gorenstein.
\end{theorem}

A codimension $2$ arithmetically Cohen-Macaulay scheme $X\subset \PP^n$ is Gorenstein if and only if it is a complete intersection (\cite[Example 4.1.11({c})]{mig}). Thus we cannot require that $X$ is Gorenstein for $x=2$ and exclude the complete intersections.  

There are classical non-extendability results for varieties of degree $3$ (\cite{XXX}) and degree $4$ (\cite{sw}), which come from a complete classification of all such low degree varieties. With some assumptions on the singularities  there are classifications for degree $5$ (\cite{ti}) and up to degree $10$ (\cite{fl}). No such results should be true in general, without also restricting the singularities of the extensions, not just of the variety. Very interesting non-extension results use the normal bundle $N_X$ and $h^0(N_X(-1))-n-1$ is conjecturally an upper bound for the number of extensions, as parenthetically asked in \cite[middle of the first page]{cds}, which would extend a non-extendability theorem of S. Lvovsky (\cite{lv}), but we think that it is essential to assume something also on the singularities of the extensions. By \cite{f} or \cite[Theorem 2]{coa2}  it is not sufficient to assume that the singularities are locally complete intersections.
\section{The proof}
\begin{proof}[Proof of Theorem \ref{aa1}:] Any extension of an arithmetically Cohen-Macaulay scheme of positive dimension is arithmetically Cohen-Macaulay and hence all its extensions are arithmetically Cohen-Macaulay. We start with an integral arithmetically Cohen-Macaulay $X$ which is not a cone. We will extend
it to $W\subset \PP^{n+1}$ with $W$ integral and not a cone.  $W$ would be arithmetically Cohen-Macaulay and hence the construction may be iterated starting with $W$. After $r$ steps we get that $X$ is $r$-extendable. Since this is true for all $r>0$, $X$ is infinitely extendable in the sense of our Definition \ref{a1}. For arithmetically Cohen-Macaulay schemes to be a complete intersection is a property of their minimal free resolution. Thus an extension $Y_1$ of an arithmetically Cohen-Macaulay scheme $Y$ of positive dimension is a complete intersection if and only if $Y$ is a complete intersection. Thus in each step we may omit any check that the extension is not a complete intersection.

\quad (a) Assume $x=2$ with $n\ge 3$. Fix positive integers $s\ge 2$, $n_{1i}$, $1\le i\le s$, and $n_{2i}$, $1\le i\le s-1$, such that $\sum _{i=1}^{s-1} n_{2i} = \sum _{i=1}^{s} n_{1i}$. We call $\{n_{ij}\}$ these data $s$, $n_{1i}$ and $n_{ij}$. Let $\Uu_n(n_{ij})$ denote the set of all codimension $2$ arithmetically Cohen-Macaulay schemes with free resolution
\begin{equation}\label{eqp1}
0 \to \oplus _{i=1}^{s-1} \Oo _{\PP^n}(-n_{2i}) \to \oplus _{i=1}^{s} \Oo_{\PP^n}(-n_{1i}) \to \Ii _X\to 0
\end{equation}
By \cite[Th. 2]{e} the set $\Uu_n(n_{ij})$ is a non-empty open and irreducible subset of the Hilbert scheme $\mathrm{Hilb}(\PP^n)$ of $\PP^n$ of dimension
\begin{align}\label{eqp2}
&1+ \sum_{n_{2i}\ge n_{1j}} \binom{n_{2i}-n_{1j}+n}{n} +\sum _{n_{1j}\ge n_{2i}} \binom{n_{1j}-n_{2i}+n}{n} \notag \\
&  -\sum _{n_{2i}\ge n_{2j}}\binom{n_{2i}-n_{2j}+n}{n} -\sum_{n_{1j}\ge n_{1i}} \binom{n_{1j}-n_{1i}+n}{n}
\end{align}

For instance the case $s=2$ corresponds to complete intersection codimension $2$ schemes.  
 Call $\psi _n(n_{ij})$ the integer appearing in \eqref{eqp2}. For instance
 $\psi_n(n_{ij}) =1+6(n+1)-4-9 =6n-6$ for $s=3$, $n_{11}=n_{12}=n_{13}=2$ and $n_{21}=n_{22} =3$, the case corresponding to codimension $2$ degree $3$ varieties
 (fully classified by \cite{XXX} and which are always cones for $n\ge 5$).

 See $\PP^n$ as a hyperplane $H$ of $\PP^{n+1}$. 
Let $\Uu _{n+1,H}(n_{ij})$ the set of all $W\in \Uu_{n+1}(n_{ij})$ such that no irreducible component of $W_{\mathrm{red}}$ is
contained in
$H$. Note that $W\cap H\in \Uu _n(n_{ij})$ and that $W$ is an extension of $W\cap H$ for each $W\in \Uu _{n+1,H}(n_{ij})$. Fix any $X\subset H$ such that $X\in
\Uu_n(n_{ij})$. Let $\rho: \Uu _{n+1,H}(n_{ij})\to  \Uu _n(n_{ij})$ be the morphism defined by the formula $W\mapsto W\cap H$. There are exactly $\infty^{n+1}$ $W\in \Uu _{n+1,H}(n_{ij})$ such that $W\cap H = X$ and $W$ is a cone with
vertex containing some $p\in \PP^{n+1}\setminus H$. Thus $\rho$ is a surjective morphism between integral varieties. Thus if $W$ is general in $\Uu _{n+1}(n_{ij})$, then $W\in
\Uu_{n+1,H}(n_{ij})$ and $W\cap H$ is general in $\Uu_n(n_{ij})$. Recall that for each $X\in \Uu_n(n_{ij})$ the set of cones belonging to $\rho^{-1}(X)$ has dimension $n+1$. Thus to prove that a general $X\in \Uu_n(n_{ij})$ may be
extended to a general $W\in \Uu_{n+1}(n_{ij})$ which is not a cone with vertex contains some $p\in \PP^{n+1}\setminus H$ it
is sufficient to prove that $\psi_{n+1}(n_{ij}) \ge \psi _n(n_{ij})+n+2$. If we do this for all large $n$ then we get an example
for the case codimension $2$ of Theorem \ref{aa1}. We only do one example of $s$ and $n_{ij}$, which satisfies these properties
for all $n\ge 3$.

 Fix integers $s\ge 3$ and an integer $c\ge 2$. Set $n_{1i}:= (s-1)c$ for all $i$ and $n_{2j}:= sc$ for all $j$.
We have $\psi _n(n_{ij}) =s(s-1)\binom{x+n}{n} -(s-1)^2 -s^2+1$. Thus $\psi_{n+1}(n_{ij}) -\psi _n(n_{ij}) = s(s-1)\binom{x+n}{n+1} \ge s(s-1)(n+2)/2$. For $n=3$ this case corresponds to an integral curve because its numerical character is connected (\cite[Th. 2.5]{gp}) and it is even smooth (\cite{sau}). Hence inductively we get an integral $X$, as claimed by the theorem.

\quad (b) Now we do the codimension $3$ case. In this part we get $X$ which is arithmetically Gorenstein. We use a paper by J.
O. Kleppe and R.-M.  Mir\'{o}-Roig (\cite{kmr}), which of course use the classical Buchsbaum-Eisenbud's description of codimension $3$ Gorenstein local rings (\cite{be}).
At the end of \cite[\S 1]{kmr} they introduce the integers appearing in their computation of the dimension of the Hilbert scheme of arithmetically Gorenstein codimension $3$
subschemes of $\PP^n$, $n\ge 4$, with prescribed numerical invariants $f$ (or $e$ with $f=e+n+1$), $r\ge 3$, $n_{2i}$ and $n_{1j}$ with $n_{11}\le \cdots \le n_{1r}$, $n_{21}\ge \cdots \ge n_{2r}$ with $n_{1i} $ for all $i$. The dimension of this iirreducible component of the Hilbert scheme is given in \cite[Remark 2.8]{kmr}. Call $\phi _n(n_{ij})$ this dimension. By \cite[Remark 2.8]{kmr} we have
\begin{align}
& \phi _n(n_{ij}) = \sum _{1\le i<j\le r} \binom{-n_{1i}+n_{2j}+n}{n}\\
&-\sum _{1\le i, j\le r}\binom{-n_{1i}+n_{1j}+n}{n} + \sum _{1\le i\le j\le r} \binom{n_{1i}-n_{2j}+n}{n} \notag
\end{align}
We only give one case, the case with $r=7$, $n_{11}=n_{12}=n_{13}=n_{14}=n_{15} =4$, $n_{16}=n_{17}=n_{26}=n_{27} =5$ and $n_{21}=n_{22}=n_{23}=n_{24}=n_{25}=6$. The big advantage of this case is that we do not need to check the existence of an integral $X$, since in the starting case, the case $n=4$, we may take $X$ smooth and irreducible (\cite[Example 2.9]{kmr}.
We have 
\begin{align}
&\phi _n(n_{ij}) = 10\binom{-4+6+n}{n} +10\binom{-4+5+n}{n} -25 -4\notag\\
&-10\binom{-4+5+n}{n}+3 = 10\binom{n+2}{2}  -26 \notag
\end{align}
Thus $\phi_{n+1}(n_{ij})-\phi _n(n_{ij}) = 10(n+2)>n+2$.

\quad ({c}) Take any $x\ge 4$ and any $n \ge x+1$. By step (b) there are integral, non-degenerate, arithmetically Cohen-Macaulay and Gorenstein varieties $Y_i\subset \PP^{n+i}$, $i\ge 0$, such that each $Y_i$, $i>0$, is an extension (not cone-like) of $Y_{i-1}$ and no $Y_i$ is a complete intersection. Fix general quadric hypersurfaces $Q_1(0),\dots ,Q_{x-3}(0)\subset \PP^n$. Let $Q_j({r}) \subset \PP^{n+r}$ be an $r$-step extension of $Q_j(0)$, $1\le j\le x-3$. Since $Q_j(0)$ is general, it is smooth. Thus we may find as extensions smooth quadric hypersurfaces. Set $X_i:= Y_i\cap Q_1(i)\cap \cdots \cap Q_{x-3}(i)$ and $X:= X_0$.
Each $X_i\subset \PP^{n+i}$ is non-degenerate, integral (by Bertini's theorem) and  arithmetically Gorenstein. Since a minimal system of generators of the homogeneous ideal of $X_i$ is obtained adding $Q_1(i),\dots ,Q_{x-3}(i)$ to a minimal system of generators of $Y_i$, $X_i$ is not a complete extension. Fix an integer $i>0$ and let $H\subset \PP^{n+i}$ be the hyperplane generated by $Y_{i-1}$. $H$ is the hyperplane generated by $X_{i-1}$. Assume the existence of $p\in \PP^{n+i}\setminus H$ such that
$X_i$ is a cone with vertex $p$ and $X_{i-1}$ as a basis. Since $X_{i-1}$ spans $H$, $\PP^{n+i}$ is the Zariski tangent space to $X_i$ at $p$. Thus each quadric hypersurface containing $X_i$ is a cone with vertex $p$. Thus $Q_1(i)$ is a cone, a contradiction.
\end{proof}


\bibliographystyle{amsplain}

\begin{thebibliography}{99}


\bibitem{bvdv} W.  Barth and A. Van de Ven,  A decomposability criterion for algebraic 2-bundles on projective spaces, Invent. Math. 25 (1974), 91--106.

\bibitem{be} D. Buchsbaum and D. Eisenbud, Algebra structures for finite free resolutions, and some structure
theorems for ideals of codimension $3$, Amer. J. Math. 99 (1977) 447--485.


\bibitem{cds} C. Ciliberto, T. Dedieu and E. Sernesi, Wahl maps and extensions of canonical curves and K3 surfaces,  J. Reine Angew. Math. 761 (2020), 219--245.

\bibitem{coa1}  I. Coand\u{a},  Infinitely stably extendable vector bundles on projective spaces,  Arch. Math. 94 (2010), 539--545.

\bibitem{coa2} I. Coand\u{a}, A simple proof of Tyurin's Babylonian tower theorem, Comm. Algebra 40 (2012), no. 12, 4668--4672.



\bibitem{ct} I. Coand\u{a}, and G. Trautmann, The splitting criterion of Kempf and the Babylonian tower theorem, Comm. Algebra 34 (2006), 2485--2488.

\bibitem{e} G. Ellingsrud, Sur le sch\'{e}ma de Hilbert des vari\'{e}t\'{e}s de codimension 2 dans $\mathbf{P}^e$ \`{a} c\^{o}ne de Cohen-Macaulay,  Ann. Sci. \'{E}cole Norm. Sup. (4) 8 (1975), no. 4, 423--431. 

\bibitem{fl}  M. L. Fania, and E. L. Livorni,  Degree ten manifolds of dimension n greater than or equal to 3, Math. Nachr. 188 (1997), 79--108.

\bibitem{f} H. Flenner, Babylonian tower theorems on the punctured spectrum, Math. Ann. 271 (1985), 153--160.

\bibitem{gp} L. Gruson and C. Peskine,Genre des courbes de l'espace projectif, in: Algebraic geometry.
Troms\o \  1977. Leer. Notes Math. 687, 31--59. Berlin, Heidelberg, New York: Springer 1978

\bibitem{kmr} J. O. Kleppe and R.-M. Mir\'{o}-Roig, The dimension of the Hilbert scheme of Gorenstein codimension 3
subschemes, J. Pure Appl. Algebra 127 (1998), 73--82.



\bibitem{lv} S. Lvovsky, Extensions of projective varieties and deformations, I, II, Michigan Math. J. 39 (1992), 41--51, 65--70.

\bibitem{mig} J. Migliore, Introduction to liaison theory and deficiency modules, Birk\"{a}user, Boston-Basel-Berlin, 1998.



\bibitem{pt} I. B.  Penkov and A.S.Tikhomirov,  On the Barth--Van de Ven--Tyurin--Sato theorem, (Russian) Mat. Sb. 206 (2015), no. 6, 49--84; translation in Sb. Math. 206 (2015), no. 5-6, 814--848. 

\bibitem{s1}  E. Sato, On the decomposability of infinitely extendable vector bundles on projective spaces and Grassmann varieties. J. Math. Kyoto Univ. 17 (1977), 127--150. 

\bibitem{s2} E. Sato,  The decomposability of an infinitely extendable vector bundle on the projective space, II. In: International Symposium on Algebraic Geometry. Kyoto University. Kinokuniya Book Store: Tokyo, pp. 663--672, 1978. 

\bibitem{sau} T. Sauer, Smoothing projectively Cohen-Macaulay space curves,
Math. Ann. 272 (1985), no. 1, 83--90. 

\bibitem{sw}  H. P. F. Swinnerton-Dyer, An enumeration of all varieties of degree $4$, Amer. J. Math. 95 (1973), 403--418.

\bibitem{ti}  A. L. Tironi,  Normal projective varieties of degree 5, Comm. Algebra 42 (2014), no. 10, 4322--4332.

\bibitem{t} A. N. Tyurin, A. N. Finite dimensional vector bundles over infinite varieties. Math. USSR Izv. 10 (1976), 1187--1204. 


\bibitem{XXX} XXX = A. Weil, Correspondence by XXX, Amer. J. Math. 79 (1957), 951--952; reprinted in Oeuvres scientifiques/collected papers. II. 1951--1964, Appendix I, pp. 555--556.

\bibitem{z1} F. L. Zak, Projections of algebraic varieties, Mat. Sb. (N.S.) 116(158) (1981),
no. 4, 593--602, 608.

\bibitem{z2} F. L. Zak, Some properties of dual varieties and their applications in projective geometry, Algebraic geometry (Chicago, IL, 1989), 273--280, Lect. Notes in Math. 1479, Springer, Berlin, 1991.

\end{thebibliography}
\providecommand{\bysame}{\leavevmode\hbox to3em{\hrulefill}\thinspace}
\providecommand{\MR}{\relax\ifhmode\unskip\space\fi MR }
\providecommand{\MRhref}[2]{%
  \href{http://www.ams.org/mathscinet-getitem?mr=#1}{#2}
}
\providecommand{\href}[2]{#2}
\providecommand{\bysame}{\leavevmode\hbox to3em{\hrulefill}\thinspace}

\end{document}